\documentclass[12pt,a4paper]{article}

\usepackage{amssymb, amsmath, amsthm}

\usepackage{fullpage}
\usepackage[english]{babel}
\usepackage[pdftex]{graphicx}

\usepackage{hyperref}

\newtheorem{thm}{Theorem}

\newtheorem{rem}{Remark}
\newtheorem{cor}{Corollary} 

\theoremstyle{definition}
\newtheorem{df}{Definition}

\def\Xint#1{\mathchoice
	{\XXint\displaystyle\textstyle{#1}}%
	{\XXint\textstyle\scriptstyle{#1}}%
	{\XXint\scriptstyle\scriptscriptstyle{#1}}%
	{\XXint\scriptscriptstyle\scriptscriptstyle{#1}}%
	\!\int}
\def\XXint#1#2#3{{\setbox0=\hbox{$#1{#2#3}{\int}$ }
		\vcenter{\hbox{$#2#3$ }}\kern-.6\wd0}}

\def\dashint{\Xint-}

\title{The Bushell-Okrasi\'nski inequality}
\author{\L ukasz P\l ociniczak\thanks{Faculty of Pure and Applied Mathematics, Wroc{\l}aw University of Science and Technology, Wyb. Wyspia{\'n}skiego 27, 50-370 Wroc{\l}aw, Poland}$\;^,$\footnote{Email: lukasz.plociniczak@pwr.edu.pl}}
\date{\textit{Dedicated to the memory of \\ Peter J. Bushell (1934-2020) and Wojciech Okrasi\'nski (1950-2020)}}

\begin{document}
	\maketitle
	
	\begin{abstract}
		We present an expository account of the Bushell-Okrasi\'nski inequality, the motivation behind it, its history, and several generalizations. This inequality originally appeared in studies of nonlinear Volterra equations but very soon gained interest of its own. The basic result has quickly been generalized and extended in different directions strengthening the assertion, generalizing the kernel and nonlinearity, providing the optimal prefactor, finding conditions under which it becomes an equality, and formulating variations valid for other than Lebesgue integrals. We review all of these aspects. \\
		
		
	\end{abstract}

\section{Introduction}
Analysis is full of integral inequalities of many types and utility. Every young adept of the art has to learn and efficiently use results of Cauchy, Schwarz, H\"older, Jensen, Minkowski, Young, Sobolev, Poincar\'e, Friedrichs, Hardy, and Chebyshev to name only a few classics. There is another type of inequality that resides somewhere between H\"older's and Jensen's. A result that can be thought as a strengthening of the Chebyshev inequality. The Bushell-Okrasi\'nski inequality, which in one of its basic forms for positive and increasing $f$ can be stated as
\begin{equation}
	\int_0^x (x-s)^{\alpha-1} f(s)^\alpha ds \leq \left(\int_0^x f(s)ds\right)^\alpha, \quad \alpha \geq 1,
\end{equation}
has been discovered in 1990 by Peter Bushell and Wojciech Okrasi\'nski and published in their work on nonlinear Volterra integral equations \cite{Bus90}. The above result has quickly been included in Bullen's "Dictionary of inequalities" (\cite{Bul15}, p.35). Subsequently, many authors proceeded to investigate it further by relaxing assumptions, strengthening the claim, finding optimal constants, generalizing to other than power type kernels and nonlinearities, and translating it to fuzzy integrals. In this paper we will look closely on historical development of this inequality and review some of its generalizations.

This review is structured as follows. First, we give some motivations behind Bushell-Okrasi\'nski inequality and present its original proof. Then, we discuss Wolfang Walter's conjectures and their resolution in a joint work with Weckesser. Additionally, we present several different generalizations of the original result. We end the paper with a short detour into the land of fuzzy integrals which also can enjoy some types of Bushell-Okrasi\'nski inequality. 

\section{The original Bushell-Okrasi\'nski inequality}
The main motivation behind the original Bushell-Okrasi\'nski inequality was a study of nonlinear Volterra equations of the form
\begin{equation}
\label{eqn:NonlinearVolterra}
	u(x) = \int_0^x k(x-s) g(u(s)) ds,
\end{equation}
that arise in many important applications in porous media \cite{Kin88,Kne77} or shocks \cite{Kel81}. It is instructive to make a trip to the field of hydrology and see how the above Volterra integral equation can appear as a model of moisture imbibition. Suppose that an initially dry and half-infinite porous medium is subjected to water at $x = 0$. Then, in the absence of gravity, the capillary action is the only factor that drives the evolution of the moisture $\theta=\theta(x,t)$ (that is, the percentage of representative volume filled with water). The mass conservation gives then the nonlinear diffusion equation known as the Richards equation with the nonlinear diffusivity $D(\theta)$ \cite{Bea13},
\begin{equation}
\begin{cases}
	\dfrac{\partial\theta}{\partial t} = \dfrac{\partial}{\partial x}\left(D(\theta) \dfrac{\partial \theta}{\partial x} \right), & x > 0, \quad t > 0 \\
	\theta(x,0) = 0, \\
	\theta(0,t) = 1, & \lim\limits_{x\rightarrow 0^+} -D(\theta) \dfrac{\partial \theta}{\partial x} = 0,
\end{cases}
\end{equation}
where for simplicity we have chosen the appropriate physical units so that the resulting problem is nondimensional. Notice the no-flux boundary condition that tells us that no water is being injected into the medium - capillary forces do all the work. It is natural, both theoretically and experimentally, to seek for the self-similar solutions of the above in the form $\theta(x,t) = v(\eta)$ with $\eta := x/\sqrt{t}$. This gives an ordinary differential equation
\begin{equation}
	-\frac{1}{2}\eta v' = (D(\theta) v')', \quad v(0) = 0, \quad (\cdot)' := \frac{d}{d\eta}.
\end{equation}
Now, we can integrate the above over $[0,\eta]$ to obtain
\begin{equation}
	\frac{1}{2}\int_0^\eta v(s)ds = D(v(\eta))v'(\eta) - D(v(0))v'(0) = D(v(\eta))v'(\eta),
\end{equation}
where the second term on the right vanishes due to the no-flux condition. Further, we can define the primitive of the diffusivity $D$
\begin{equation}
	G(u) := \int_0^u D(s) ds,
\end{equation}
and with the help of which we have $D(v)v' = (G(v))'$. Therefore, a second integration yields
\begin{equation}
	G(v(\eta)) = \frac{1}{2}\int_0^\eta (\eta-s) v(s)ds.
\end{equation}
Finally, since $D$ is positive there is a well-defined inverse $g = G^{-1}$. Setting $v = g(u)$ brings us to (\ref{eqn:NonlinearVolterra}) with a generalized kernel. For example, weakly singular kernels $k(s) \propto s^{\alpha-1}$ for some $\alpha\in (0,1)$ can arise in the study of anomalous diffusion \cite{Plo14,Plo15}.  

The flagship example of nonlinearity is the root function $g(u) = u^{1/p}$ for some $p>1$ that models the diffusivity of many porous media (this is the Brooks-Correy model of moisture transport in soil \cite{Bro64}). As can be easily observed, in this case the integral equation (\ref{eqn:NonlinearVolterra}) has a trivial solution $u\equiv 0$. More generally, for a nonlinearity satisfying $g(0) = 0$ the trivial solution is always present. However, when $g$ is \emph{non-Lipschitz} then a non-trivial solution might exist (the Lipschitz condition rules out this case). Investigating these solutions is the main objective of Bushell's and Okrasi\'nski's paper \cite{Bus90} as well as several other authors throughout the last decades (for ex. \cite{Gri81,Nie97,Myd91,Rob98,Okr91,Bus86,Bus76,Rob96,Bru04}). For a thorough review on this account in view of numerical methods the Reader is referred to \cite{Ari19}. For the aforementioned root-type nonlinearity $g(u) = u^{1/p}$ one can easily verify that a non-trivial solution of (\ref{eqn:NonlinearVolterra}) with a kernel $k(s) = s^{\alpha-1}$ and some $\alpha > 0$ is the power function
\begin{equation}
	u(x) = x^{\frac{\alpha p }{p-1}} B\left(\alpha, 1+\frac{\alpha }{p-1}\right)^\frac{p}{p-1},
\end{equation}
where $B(\cdot, \cdot)$ is the Euler beta function. There are many approaches that find the necessary and sufficient conditions on $k$ and $g$ for which (\ref{eqn:NonlinearVolterra}) has non-trivial solutions and reviewing all of them would take us too far from the main theme of this paper. However, we shortly note that Bushell's and Okrasi\'nski's approach is to use the monotone iteration method with sub- and supersolutions defined in functional cones (see for ex. \cite{Zei12}). In the main argument authors construct an \emph{explicit} solution of the following associated integral equation
\begin{equation}
	w(x) = \int_0^x g(w(s)^\alpha)^\frac{1}{\alpha} ds,
\end{equation}
and show that it can exists if and only if $1\leq \alpha < \alpha_c$ for some critical value $\alpha_c$. This assertion is then carried over to the case of (\ref{eqn:NonlinearVolterra}) and the crucial link between these two nonlinear Volterra equations is supplemented by the Bushell-Okrasi\'nski inequality.

In many talks between the author and W. Okrasi\'nski he always stressed that the inequality was just a "passing auxiliary lemma" needed to show necessary conditions for existence. Originally, W. Okrasi\'nski did not realize that it can have a value of its own. This is probably the reason that he together with P. Bushell did not try to polish the result and provide a stronger assertion. This was later done by other authors, and to some extent by Bushell, to what we will turn in next sections. Now, we present the original proof of Bushell-Okrasi\'nski inequality\footnote{In this review we will state all important results as "Theorems" in contrast, for example, with \cite{Bus90} where the BO inequality is denoted as "Lemma".}\\

\begin{thm}[The original Bushell-Okrasi\'nski inequality, Lemma 2 in \cite{Bus90}]
Let $f\in C[0,X]$, $0<X\leq 1$, be a non-decreasing and non-negative function. If $\alpha\geq 1$, then
\begin{equation}
\label{eqn:BOOriginal}
	\int_0^x (x-s)^{\alpha-1} f(s) ds \leq \left(\int_0^x f(s)^\frac{1}{\alpha}ds\right)^\alpha. 
\end{equation}
\end{thm} 
\begin{proof}
Our aim is to show (\ref{eqn:BOOriginal}) first for natural $\alpha$ by mathematical induction, then by H\"older inequality extend the result to all rational numbers, and finally by density argument arrive at $\alpha$ real. Fix $f$ as in the assumptions and define for $n\in\mathbb{N}$ let 
\begin{equation}
\label{eqn:BOIn}
	I_n(x) := \left(\int_0^x f(s)ds\right)^n - n \int_0^x (x-s)^{n-1} f(s)^n ds. 
\end{equation}
We can compute the derivative of the above quantity 
\begin{equation}
\begin{split}
	I'_n(x) &= n f(x) \left(\int_0^x f(s)ds\right)^{n-1} - n(n-1) \int_0^x (x-s)^{n-2} f(s)^n ds \\
	&= n f(x) I_{n-1}(x) + n(n-1) \int_0^x \left(f(x)-f(s)\right)(x-s)^{n-2} f(s)^{n-1} ds,
\end{split}
\end{equation}
which is valid for all $n \geq 1$. For $n=1$ we trivially have $I_1(x) = 0$, while for the next step
\begin{equation}
	I'_2(x) = 2 \int_0^x \left(f(x)-f(s)\right) f(s) ds \geq 0,
\end{equation}
because $f$ is non-decreasing. Therefore, $I_2(x) \geq I_2(0) = 0$. Now, we assume that $I_{n-1}(x) \geq 0$ for $n > 2$. We immediately have $I_n(x) \geq 0$ since manifestly $I'_n(x) \geq 0$ due to inductive assumption and non-decreasing of $f$. Therefore, $I_n(x)\geq 0$ for all $n\in\mathbb{R}$, that is
\begin{equation}
\label{eqn:BOOrginalStronger}
	n \int_0^x (x-s)^{n-1} f(s)^n ds \leq \left(\int_0^x f(s)ds\right)^n,
\end{equation}
which implies (\ref{eqn:BOOriginal}) for $\alpha = n$ with $f$ replaced by $f^{1/n}$ which is also positive and non-decreasing. 

Now, let us fix $p>1$ and take the conjugate exponent $q^{-1} = 1 - p^{-1}$. By H\"older inequality we have
\begin{equation}
	\int_0^x (x-s)^{\frac{n-1}{p}} f(s)^\frac{n}{p}f(s)^\frac{1}{q} \leq \left(\int_0^x (x-s)^{n-1}f(s)^n ds\right)^\frac{1}{p} \left(\int_0^x f(s) ds\right)^\frac{1}{q}.
\end{equation}
If we now put $\alpha = 1+ (n-1)/p$, due to arbitrariness of $p$ we obtain (\ref{eqn:BOOriginal}) for any $\alpha\in\mathbb{Q}$. From the density of rational numbers in $\mathbb{R}$ we obtain the Bushell-Okrasi\'nski inequality for all real $\alpha \geq 1$. This concludes the proof.
\end{proof}

\begin{rem}
The original result, that is Lemma 2 in \cite{Bus90} also contains the following inequality 
\begin{equation}
	\left(\frac{\beta-\alpha}{\beta-1}\right)^{\beta-1} \left(\int_0^x f(s)^\frac{1}{\beta}ds\right)^\beta \leq \int_0^x (x-s)^{\alpha-1} f(s)ds,
\end{equation}
for $\beta > \alpha$ without the requirement on $f$ to be non-decreasing. However, the proof is a simple consequence of H\"older's inequality and thus we omit it here. This justifies the claim that (\ref{eqn:BOOriginal}) is a \emph{reverse H\"older type inequality} (see \cite{Bar97}). 
\end{rem}

We close this section with some remarks concerning nonlocal operators, in particular fractional integrals. For a comprehensive treatment of this subject the Reader is invited to consult \cite{Kil06}. The notion of generalizing derivatives to not necessarily integer order has been present in mathematics since the beginning of the calculus itself (an interesting historical account can be found in \cite{Mil93}). Many different approaches have been undertaken and culminated in the definition of the Riemann-Liouville fractional integral and derivative.

\begin{df}
The \emph{Riemann-Liouville fractional integral} of order $\alpha>0$ of a locally integrable function $f: [0, X] \mapsto \mathbb{R}$ is given by
\begin{equation}
	I^\alpha_a f(x) = \frac{1}{\Gamma(\alpha)} \int_a^x (x-s)^{\alpha-1} f(s) ds.
\end{equation}
Moreover, let $n = [\alpha]$. Then the \emph{Riemann-Liouville fractional derivative} of order $\alpha$ is defined by
\begin{equation}
	D^\alpha_a f(x) = \frac{d^n}{dx^n} I^{n-\alpha}_a f(x) = \frac{1}{\Gamma(n-\alpha)} \frac{d^n}{dx^n} \int_a^x (x-s)^{n-\alpha-1} f(s) ds.
\end{equation}
\end{df}

We immediately can notice the similarity of the fractional integral and the left-hand side of Bushell-Okrasi\'nski inequality. From (\ref{eqn:BOOriginal}) we obtain for $0<x\leq X\leq 1$ and $f$ positive non-decreasing,
\begin{equation}
	I^\alpha_0 f(x) \leq \frac{1}{\Gamma(\alpha)}\left(\int_0^x f(s)^\frac{1}{\alpha} ds\right)^\alpha \leq \frac{1}{\Gamma(\alpha)} \left(\int_0^1 f(s)^p ds\right)^\frac{1}{p}, \quad 1\leq p\leq \infty,
\end{equation}
which follows from H\"older's inequality. Therefore, Bushell-Okrasi\'nski inequality implies that the fractional integral is a bounded linear operator for such functions (for a general $L^p(0,1)$ space see \cite{Kil06}, Lemma 2.1). In fact, we will see below in (\ref{eqn:BOStronger}) that the optimal constant in the bound is equal to $1/\Gamma(1+\alpha)$. 

\section{Improvements and generalizations}

\subsection{Walter's conjectures}
As we have mentioned above, authors of the original 1990 paper \cite{Bus90} treated the Bushell-Okrasi\'nski inequality as a side lemma that was needed to investigate the non-trivial solutions of the nonlinear Volterra integral equation (\ref{eqn:NonlinearVolterra}). Very soon, in fact almost immediately after the publication, the inequality gained some attention in mathematical community. In December 1990 during the 6th International Conference on General Inequalities in Oberwolfach, Germany, Wolfgang Walter posed two conjectures related to strengthening of the original inequality (the proceedings \cite{Wal92} appeared in 1992). The first one is based on an observation that in the original proof a stronger result is obtained for $\alpha\in\mathbb{N}$ as in (\ref{eqn:BOOrginalStronger}) where the factor $n$ appears on the left-hand side. Then, in the proof for $\alpha\in\mathbb{Q}$ this fact is used only partially and the stronger inequality is lost. Walter asked whether it is possible to find a simpler proof of Bushell-Okrasinski inequality with improved assertion that for all $\alpha \geq 1$ we have
\begin{equation}
\label{eqn:BOStronger}
	\alpha \int_0^x (x-s)^{\alpha-1} f(s)^\alpha ds \leq \left(\int_0^x f(s) ds\right)^\alpha.
\end{equation}
It is also natural to ask whether the above is satisfied for $0<\alpha<1$ provided the function $f$ is nonnegative and decreasing. Additionally, the question arises whether the assumption that $f$ is defined on an interval $[0,X]$ with $X\leq 1$ is necessary at all. All of these claims has been successfully proved in a joint paper with V. Weckesser published in 1993 (the paper was submitted in November 1991 and revised in August 1992). In fact these authors proved much more general result with a proof that is based on approximations by step functions and Monotone Convergence Theorem. 

\begin{thm}[Generalized Bushell-Okrasinski inequality, Theorem 1 in \cite{Wal93}]
\label{thm:Walter}
Suppose that $f:[0,X] \mapsto [0,\infty)$ for $X > 0$, $g:[0,\infty)\mapsto [0,\infty)$, and $k \in L^1[0,X]$. Define 
\begin{equation}
\label{eqn:WalterCondition}
	K(x) := \int_0^x k(s) ds, \quad h_c(y) := g( cy) - K(c) g(y), 
\end{equation}
for $0<c\leq X$. Then, if either
\begin{equation}\tag{I}
\label{eqn:CondI}
	f \text{ is non-decreasing},\quad g \text{ is convex},\quad h_c \text{ is nonnegative and non-decreasing},
\end{equation}
or
\begin{equation}\tag{II}
\label{eqn:CondII}
	f \text{ is non-increasing},\quad g \text{ is concave},\quad h_c \text{ is nonnegative and non-increasing},
\end{equation} 
the following generalized Bushell-Okrasi\'nski inequality is satisfied
\begin{equation}
\label{eqn:BOGeneralized}
	\int_0^x k(x-s) g(f(s)) ds \leq g\left(\int_0^x f(s)ds\right), \quad 0<x\leq X.
\end{equation}
\end{thm}

Before we proceed to the proof we relate the original Walter's conjectures to the above theorem.
\begin{cor}
The inequality (\ref{eqn:BOStronger}) is satisfied for all $0<x\leq X$ with an arbitrary $X >0$ when
\begin{itemize}
	\item $f:[0,X]\mapsto [0,\infty)$ is non-decreasing and $\alpha\geq 1$, or
	\item $f:[0,X]\mapsto [0,\infty)$ is non-increasing and $0<\alpha \leq 1$. 
\end{itemize}
\end{cor}
\begin{proof}
We have $k(x) = \alpha x^{\alpha-1}$ and hence, $K(x) = x^\alpha$. Further, $g(y) = y^{\alpha}$ and thus for $0<c\leq X$ we have,
\begin{equation}
	h_c(y) = g(cy) - K(c) g(y) = (c y)^{\alpha} - c^\alpha y^{\alpha} = 0,
\end{equation}
and the condition for $h_c$ is satisfied both in (\ref{eqn:CondI}) and (\ref{eqn:CondII}). Since the power function $g$ is concave for $0<\alpha\leq 1$ and convex for $\alpha \geq 1$ the proof is completed by the use of Theorem \ref{thm:Walter}. 
\end{proof}
We can now proceed to the proof of Walter and Weckesser's result.
\begin{proof}(of Theorem \ref{thm:Walter}). 
We will focus only on (\ref{eqn:CondI}) case, the proof of the other is similar. Due to Beppo Levi's Theorem of Monotone Convergence, it is sufficient to consider (\ref{eqn:BOGeneralized}) for $f$ which are step functions
\begin{equation}
	f(x) = \sum_{i=1}^{n-1} a_i \chi_{[x_{i-1},x_i)}(x) + a_n \chi_{[x_{n-1},x_n]}(x),
\end{equation} 
for a partition
\begin{equation}
	0 = x_0 < x_1 < x_2 < ... < x_n = x,
\end{equation}
where $\chi_A$ is a characteristic function of a measurable set $A$. If the integral on the left in (\ref{eqn:BOGeneralized}) is denoted by $L_n$, a simple calculation gives
\begin{equation}
	L_n = \sum_{i=1}^n \left(K(x-x_{i-1})-K(x-x_i)\right)g(a_i).
\end{equation}
Similarly, 
\begin{equation}
	I_n := \int_0^x f(s) ds = \sum_{i=1}^n a_i (x_i-x_{i-1}),
\end{equation}
and, hence, the right-hand side of (\ref{eqn:BOGeneralized}), denoted by $R_n$, satisfies $R_n = g(I_n)$. Now, notice that for $n=1$, that is for constant functions, we have $x_1 = x$, and
\begin{equation}
	L_1 - R_1 = g(a_1) \left(K(x) - K(0)\right) - g(a_1 x) = h_{a_1}(x) \geq 0,
\end{equation}
by the assumption on $h_c$. We can now utilize mathematical induction. Assume that $L_n \leq R_n$. We claim that this inequality holds then for $n+1$. Without any loss of generality we can assume that the $n+1$-th step of the function $f$ can arise as a partition of the interval $[x_{n-1}, x_n]$. Pick any $0<c<x_{n}-x_{n-1}$ and $y\geq 0$. Now, the update to the non-decreasing step function $f$ is given by
\begin{equation}
	f(x) = \sum_{i=1}^{n-1} a_i \chi_{[x_{i-1},x_i)}(x) + a_n \chi_{[x_{n-1},x_n-c)} + (a_n+y) \chi_{[x_{n}-c,x_n]}.
\end{equation}
It is now straightforward to calculate the relevant integrals. It follows that only the last interval makes the difference, that is since we still have $x_n = x$, and
\begin{equation}
	\Delta L_{n+1} := L_{n+1} - L_n = \left(g(a_n +y) - g(a_n)\right) \int_{x-c}^x k(x-s)ds = \left(g(a_n +y) - g(a_n)\right) K(c). 
\end{equation}
Similarly,
\begin{equation}
	\Delta R_{n+1} := R_{n+1} - R_n = g(I_n+ cy) - g(I_n).
\end{equation}
Since, $g$ is convex and trivially $a_n c \leq I_n$ we further have $\Delta R_{n+1} \leq g(c (a_n+y)) - g(c a_n)$. Next, by the fact that $h_c$ is non-decreasing 
\begin{equation}
	\Delta L_{n+1} - \Delta R_{n+1} \leq \left(g(a_n +y) - g(a_n)\right) K(c) - g(c (a_n+y)) + g(c a_n) = h_c(a_n) - h_c(a_n+y) \leq 0.
\end{equation}
Therefore, by the inductive assumption that $L_n \leq R_n$ we have $L_{n+1} = L_n + \Delta L_{n+1} \leq R_{n} + \Delta R_{n+1} = R_{n+1}$ and the proof is complete. 
\end{proof}

As we have seen, the proof of Walter and Weckesser is of completely different nature than Bushell and Okrasi\'nski's. It can be regarded as elementary that allows for a substantial improvement of the claim. We have seen in the above corollary that taking $K$ and $g$ as power functions the general inequality (\ref{eqn:BOGeneralized}) reduces to the stronger version of the original one (\ref{eqn:BOStronger}). Quite recently, T. Ma{\l}olepszy and J. Matkowski, asked a somewhat reverse question: is this the only choice that yields Bushell-Okrasi\'nski inequality? (see \cite{Mal19}). They prove several results concerning that topic. One of which states that assuming $X>1$, (\ref{eqn:CondI}), and $K(x) = x^\alpha$, the only choice for the other function is very restricted, that is $g(y) = g(1) y^\alpha$ or $g \equiv 0$ for $\alpha \geq 1$. Interestingly, for $0<\alpha <1$ the only allowed choice is a trivial $g$. 

\subsection{Equality in (\ref{eqn:BOOrginalStronger})}
Having an inequality of the type (\ref{eqn:BOStronger}) there naturally arises a question about its sharpness. This quickly can be answered positively by taking a constant function $f$ for which the inequality becomes an equality. But is this the only case when it occurs? It was shown in \cite{Wal93} that (notice that here $x=1$)
\begin{equation}
	\label{eqn:WalterConjectureSolved}
	\alpha \int_0^1 (1-t)^{\alpha-1} f(t)ds = \int_0^1 f(t) dt \quad \iff \quad f \equiv \text{const.}
\end{equation}
Encouraged by this example, W. Walter asked whether the same conclusion holds true for the Bushell-Okrasi\'nski inequality
\begin{equation}
	\label{eqn:WalterConjecture}
	\alpha \int_0^1 (1-t)^{\alpha-1} f(t)^\alpha ds = \left(\int_0^1 f(t) dt\right)^\alpha \quad \iff \quad f \equiv \text{const.}
\end{equation}
It is relatively easy to prove that an equality in (\ref{eqn:BOOrginalStronger}) occurs only for constant $f$ when $\alpha \in\mathbb{N}$. The proof follows the same route as the original one by Bushell and Okrasi\'nski for their inequality. One has just to inductively verify conditions for which $I_n$ defined in (\ref{eqn:BOIn}) is equal to $0$. This method was generalized P. Bushell and A. Carbery in \cite{Bus01}. They have proved that the equality in (\ref{eqn:BOOrginalStronger}) for all $\alpha\geq 1$ (in fact, in some generalized version of it) occurs only if for some $0\leq x_0 < X$ we have
\begin{equation}
	f(x) = 
	\begin{cases}
		0, & 0 < x \leq x_0, \\
		C, & x_0 < x \leq X,
	\end{cases}
\end{equation}
where $C>0$ is a constant. This resolves the open problem posed by Walter. 

\subsection{Reversed Jensen type inequalities}
In subsequent years following \cite{Bus90} and \cite{Wal92} several other generalization and improvements of Bushell-Okrasi\'nski inequality appeared in the literature. For example, Y. Egorov in 2000 gave another elementary functional-analytic proof of Walter's conjecture \cite{Ego00} in the case of continuous functions for a slightly stronger inequality as in (\ref{eqn:BOOrginalStronger}). 
%
%
%
P. Bushell himself went further into the direction of investigating the reversed Jensen inequality. In a paper with A. Carbery \cite{Bus01} they stated the following result (actually, they have proved a much general inequality). The proof is different than Walter and Weckesser's and below we present its key features. 

\begin{thm}[Corollary 2 in \cite{Bus01}]
Let $f$ be non-decreasing, positive function on $[0,X]$ and $g$ positive and convex with $g(0) = 0$. For any positive and integrable function $k$ define
\begin{equation}
	K(x) = \int_0^x k(s) ds, \quad 0<x\leq X.
\end{equation}
Further, suppose that
\begin{equation}
\label{eqn:BushellCondition}
	g\left(\frac{y}{c} \right) K(cx) \leq g(y) K(x), \quad 0<x<X, \quad y >0, \quad 0\leq c<1.
\end{equation}
Then,
\begin{equation}
\label{eqn:Bushell}
	\int_0^x k(x-s) g(f(s)) ds \leq K(x) g\left(\frac{1}{x}\int_0^x f(s) ds\right), \quad 0<x \leq X. 
\end{equation}
\end{thm}
\begin{proof}
Fix $x\in (0,X]$. It will be easier to work with a \emph{non-increasing} function $f$, because then $f(x-s)$ is \emph{non-decreasing}, and hence (\ref{eqn:Bushell}) is equivalent to 
\begin{equation}
\label{eqn:BushellNonInc}
		\int_0^x k(s) g(f(x-s)) ds \leq K(x) g\left(\frac{1}{x}\int_0^x f(s) ds\right),
\end{equation}
after substitution $s \mapsto x-s$. Henceforth, we assume that $f$ is non-increasing and we the claim the validity of the above inequality. 

Take any $\epsilon > 0$ and use (\ref{eqn:BushellCondition}) with $c = y/(y+\epsilon)$ to obtain
\begin{equation}
	g(y+\epsilon) K \left(\frac{xy}{y+\epsilon}\right) \leq g(y) K(x).
\end{equation}
By subtraction with (\ref{eqn:BushellCondition}) we arrive at
\begin{equation}
	\left(g(y+\epsilon) - g(y)\right)  K \left(\frac{xy}{y+\epsilon}\right) \leq g(y) \left(K(x) - K\left(x- \frac{x}{y+\epsilon}\epsilon\right)\right),
\end{equation}
since every convex function is differentiable $y$-a.e. we can divide by $\epsilon$ and pass to the limit obtaining
\begin{equation}
\label{eqn:BushellConditionDer}
	y g'(y) K(x) \leq x g(y) K(x).
\end{equation}
Since $f$ is non-increasing and positive, there exists $x_0 \in (0,X]$ such that $f(x) > 0$ on $[0,x_0)$ and $f(x) = 0$ on $[x_0,X)$. Therefore,
\begin{equation}
	h(x) := \frac{1}{x}\int_0^x f(s) ds \geq f(x) > 0 \quad \text{in} \quad I_0 := (0,x_0),
\end{equation}
with $h(0) = f(0)$. Then, if we define
\begin{equation}
	\Delta(x) = x g(h(x))k(x) - h(x) g'(h(x))K(x),
\end{equation}
and 
\begin{equation}
	\varphi(x) = K(x) g(h(x)) - \int_0^x k(s) g(f(s))ds,
\end{equation}
a straightforward computation yields $\varphi(0) = \varphi'(0) = 0$, and
\begin{equation}
	\varphi'(x) = k(x) f(x) \left(\frac{g(h(x))}{h(x)} - \frac{g(f(x))}{f(x)}\right) + \frac{\Delta(x)}{x h(x)} \left(h(x) - f(x)\right), \quad 0<x<x_0,
\end{equation}
and a similar expression for the interval $[x_0, X]$. Because of (\ref{eqn:BushellConditionDer}) we have $\Delta(x) \geq 0$ and hence, the second term above is non-negative. On the other hand, the function $g(y)/y$ is non-decreasing due to convexity of $g$ as can be verified by computing derivatives. Whence, we conclude that $\varphi(x) \geq 0$ which immediately proves that 
\begin{equation}
	\int_0^x k(s) g(f(s))ds \leq K(x) g(h(x)).
\end{equation}
As we mentioned before, taking any \emph{non-decreasing} function $f$, plugging $f(x-s)$ into above, and substituting $s \mapsto x-s$ yields (\ref{eqn:Bushell}) what we had to prove. 
\end{proof}
There is an interesting corollary of the above. Taking $g(y) = y^\alpha$ and $K(x) = x^\beta$ with $1\leq \alpha \leq \beta$ yields
\begin{equation}
	\beta \int_0^x (x-s)^{\beta-1}f(s)^\alpha ds \leq x^{\beta-\alpha} \left(\int_0^x f(s) ds \right)^\alpha, \quad 0<x\leq 1
\end{equation}
which is the "$\beta$-generalization" of the Bushell-Okrasi\'nski inequality. Note that Walter and Weckesser's result (\ref{eqn:BOGeneralized}) would give a weaker inequality lacking the $x^{\beta-\alpha}$ factor on the right-hand side. 

Another interesting approach to generalization of (\ref{eqn:BOOrginalStronger}) was given by S. M. Malamud in 2001 (see \cite{Mal01}). The author observed that the Bushell-Okrasi\'nski inequality is a strengthening of the classical Chebyshev's result for increasing $f \geq 0$
\begin{equation}
	\alpha \int_0^1 (1-s)^{\alpha-1} f(s)^\alpha ds \leq \left(\alpha \int_0^1 (1-s)^{\alpha-1} ds\right) \left(\int_0^1 f(s)^\alpha ds \right) =  \int_0^1 f(s)^\alpha ds.
\end{equation} 
The point is that by H\"older's inequality we always have $\int_0^1 f^\alpha ds \geq (\int_0^1 f ds)^\alpha$. Malamud proved the following general result
\begin{equation}
\label{eqn:MalamudGeneral}
	\frac{\displaystyle{\int_0^1 g(f(s)) w(s) d\Phi(s)}}{\displaystyle{\int_0^1 w(s) d\Phi(s)}} \leq g\left(\frac{\displaystyle{\int_0^1 f(s) w(s) ds}}{\displaystyle{\int_0^1 w(s) ds}} \right),
\end{equation}
when $w$ is the weight and there are certain restrictions put on $g$ and $\Phi$. The above reduces to the Bushell-Okrasi\'nski inequality for $g(x) = x^\alpha$, $\Phi(x) = 1-(1-x)^\alpha$, and $w\equiv 1$. Moreover, it can also be thought as an inverse to Jensen's inequality. Similarly to Walter and Weckesser's result, the proof proceeds by approximation of step functions. Also, a question whether the equality in the above occurs only for constant functions stays unanswered. For simplicity, we will consider only the unweighted case, i.e. $w\equiv 1$. The general inequality can be proved by the technique of approximation with step functions, however the proof is more involved and the assumptions are more restrictive. 

\begin{thm}[Theorem 2.1 in \cite{Mal01}]
Suppose that $f$ is a positive non-decreasing function, $\Phi$ is a positive function of bounded variation equal to $1$, and $g$ is positive, non-decreasing, convex, and differentiable function. Then, provided that
\begin{equation}
\label{eqn:MalamudCondition}
	g'(y) \frac{\Phi(1)-\Phi(x)}{1-x} \leq g'(y(1-x)), \quad 0<x<1, \quad 0<\lambda<\infty,
\end{equation}
we have
\begin{equation}
\label{eqn:Malamud}
	\int_0^1 g(f(s)) d\Phi(s) \leq g\left(\int_0^1 f(s) ds\right).
\end{equation}
\end{thm}
\begin{proof}
Without the loss of generality we approximate the monotone function by an increasing sequence of step-functions
\begin{equation}
	f_n(x) = \sum_{i=1}^n a_i \chi_{(\frac{i-1}{n},\frac{i}{n})}(x).
\end{equation}
For this choice, the inequality simply becomes
\begin{equation}
\label{eqn:MalamudInequalityStep}
	\sum_{i=1}^n g(a_i) \left(\Phi\left(\frac{i}{n}\right) - \Phi\left(\frac{i-1}{n}\right)\right) \leq g\left(\frac{1}{n}\sum_{i=1}^n a_i\right).
\end{equation}
We next introduce $\varphi$ as the difference between the left and right-hand side of the above as a function of the largest value of $f$, that is
\begin{equation}
	\varphi(y) := g\left(\frac{y}{n} + \frac{1}{n}\sum_{i=1}^{n-1} a_i\right) - g(y) \left(\Phi\left(1\right) - \Phi\left(\frac{n-1}{n}\right)\right) - \sum_{i=1}^{n-1} g(a_i) \left(\Phi\left(\frac{i}{n}\right) - \Phi\left(\frac{i-1}{n}\right)\right).
\end{equation}
Then, taking the derivative gives
\begin{equation}
\begin{split}
	\varphi'(y) &= \frac{1}{n} g'\left(\frac{y}{n} + \frac{1}{n}\sum_{i=1}^{n-1} a_i\right) -  g'(y) \left(\Phi\left(1\right) - \Phi\left(\frac{n-1}{n}\right)\right) \\
	&\geq \frac{1}{n} g'\left(y\left(1-\frac{n-1}{n}\right)\right) -  g'(y) \left(\Phi\left(1\right) - \Phi\left(\frac{n-1}{n}\right)\right),
\end{split}
\end{equation}
since $g$ is convex. Further, by our assumption (\ref{eqn:MalamudCondition}) we conclude that $\varphi'(y) \geq 0$. Therefore, $\varphi$ increases and we can consider (\ref{eqn:MalamudInequalityStep}) for the worst case, that is for $a_n = a_{n-1}$ (since by assumption we always have $a_n \geq a_{n-1}$). But then, by redefining the function $\varphi$ for $x=a_{n-1}$ we reduce the inequality to the case when $g_{n-1} = g_{n-2}$. Continuing in this fashion yields the obvious $\sum_{i=1}^n \left(\Phi(i/n)-\Phi((i-1)/n)\right) =1$. The general case follows from the Monotone Convergence Theorem. 
\end{proof}
We can note how the Malamud's result (\ref{eqn:Malamud}) corresponds with Walter and Weckesser's (\ref{eqn:BOGeneralized}). Recall the definition of the function $h_c(y)$ in (\ref{eqn:WalterCondition}). If we assume that it is differentiable, then it is non-decreasing when $h'_c(y) \geq 0$. But this requirement is exactly the same as (\ref{eqn:MalamudCondition}) with $c = 1-x$ and $k(t) = \Phi'(1-t)$ provided the latter derivative exists. Therefore, we can think that Walter and Weckesser's result requires less regularity than Malamud's for the generalized Bushell-Okrasi\'nski inequality to hold. Note, however, that on the other hand (\ref{eqn:MalamudGeneral}) is more general than (\ref{eqn:BOGeneralized}). 

We end this section with mentioning some other approaches to generalizing Bushell-Okrasi\'nski inequality. In 1995 H. Heinig and L. Maligranda proved that for $f$, $\Phi$ positive and non-decreasing with $\lim\limits_{s\rightarrow a^+} \Phi(s) = 0$ it holds that
\begin{equation}
	\int_a ^b f(b-s)^\alpha d\left(\Phi(s)\right)^\alpha \left(\leq \int_a^b f(b-s) d\Phi(s)\right)^\alpha, \quad \alpha \geq 1,
\end{equation}
which is (\ref{eqn:BOOrginalStronger}) for $\Phi(s) = s$, $a=0$, and $b = x$. Note that the Malamud's inequality (\ref{eqn:Malamud}) includes this case as (\ref{eqn:MalamudCondition}) is satisfied since $g(s) = s^\alpha$ is convex. Further generalizations had been given in \cite{Bar97}.

\section{Bushell-Okrasi\'nski inequality for fuzzy integrals}
Lately, a number of researchers have initiated the programme of extending the Bushell-Okrasi\'nski inequality onto some other than Lebesgue types of integrals. In 2008 a Sugeno type fuzzy integral has been considered by H. Rom\'an-Flores, A. Flores-Franuli\v c, and Y. Chalco-Cano \cite{Rom08}. In order to present this interesting result first we have to introduce some concepts concerning fuzzy measures (for a comprehensive treatment see \cite{Wan13}). 

\begin{df}
Let $\Sigma$ be the $\sigma$-algebra of subsets of $\mathbb{R}$. Then, a function $\mu: \Sigma \mapsto [0,\infty]$ is a \emph{fuzzy measure} if
\begin{itemize}
	\item $\mu(\emptyset) = 0$, 
	\item it is monotone,
	\item it is continuous from above and below. 
\end{itemize}
\end{df}

In particular, the difference with the classical measure is the fact that in the fuzzy setting we relax the requirement of additivity in favour for monotonicity from both sides. If $f$ is a non-negative real-valued function we define its $\alpha$-level set by $\{f > \alpha\} := \{x\in\mathbb{R}: \, f(x) > \alpha\}$ with $\alpha > 0$. Moreover, if $\mu$ is a fuzzy measure we define
\begin{equation}
	\mathcal{F}(\mathbb{R}) := \{f:\mathbb{R} \mapsto [0,\infty): \, f \text{ is measurable}\}
\end{equation}
This lets us define the Sugeno fuzzy integral.
\begin{df}[Sugeno integral \cite{Sug74}]
Let $\mu$ be the fuzzy measure on $(\Sigma, \mathbb{R})$. For $f\in \mathcal{F}$ and $A\in \Sigma$ the Sugeno integral (or fuzzy integral) is defined as
\begin{equation}
	\dashint_A f d\mu = \sup_{\alpha \geq 0} \left[ \min \left(\alpha, \mu(A \cap \{f\geq \alpha\})\right) \right]
\end{equation}
\end{df}
It is interesting to observe that Sugeno integrals do not enjoy some properties of the Lebesgue integrals. For example, they are not linear operators. However, many types of inequalities can be proved also for Sugeno integrals. For example, A. Flores-Franuli\v c and H. Rom\'an-Flores \cite{Flo07} showed the following Chebyshev inequality for strictly increasing continuous functions 
\begin{equation}
	\dashint_0^1  f g d\mu \geq \left(\dashint_0^1 f d\mu\right)\left(\dashint_0^1 g d\mu\right),
\end{equation}
where $\mu$ is the Lebesgue measure. In similar spirit, K. Sadarangani and J. Caballero \cite{Cab10} proved the following type of Chebyshev inequality for Sugeno integrals
\begin{equation}
	\mu\left(x\in A: f(x) > \alpha \right) \leq \frac{1}{\alpha^2} \dashint_A f^2 d\mu, \quad 0<\alpha\leq 1,
\end{equation}
for $\mu: \sigma \mapsto [0,1]$ being a fuzzy measure and positive $f\in\mathcal{F}$. It is important to note that the above inequality is valid if and only if when $0<\alpha\leq 1$. The proof is of of completely different nature than in the Lebesgue case since one cannot utilize the linearity of the integral operator. The Bushell-Okrasi\'nski type inequality is also valid for Sugeno integrals. In \cite{Rom08} authors showed that for positive, continuous, and increasing function $f$ we have
\begin{equation}
\label{eqn:BOSugeno}
	\alpha \dashint_0^1 s^{\alpha-1}f(s)^\alpha ds \geq \left(\dashint_0^1 f(s)ds\right)^\alpha, \quad \alpha \geq 2.
\end{equation}
The proof starts with the aforementioned fuzzy Chebyshev inequality and utilizes a number of techniques from fuzzy measure theory. We omit it since it is out of scope of our review. Note that the above is valid for $\alpha\geq 2$ and, surprisingly, the inequality is reversed in contrast with the result for Lebesgue integrals. However, as was shown recently by D. Hong in 2020, the above formulation of the inequality is not optimal \cite{Hon20}. Instead, with the above assumptions we have the following
\begin{equation}
\label{eqn:BOSugenoStronger}
	\left(\dashint_0^1 s^{\alpha-1} ds\right)^{-1} \dashint_0^1 s^{\alpha-1}f(s)^\alpha ds \geq \left(\dashint_0^1 f(s)ds\right)^\alpha, \quad \alpha \geq 1,
\end{equation}
where now we allow for the whole range of $\alpha$. Notice the constant in the parenthesis above. For the Lebesgue integral it would equal $\alpha$. However, as noted in \cite{Hon20}, for Sugeno case it is always smaller or equal to it (for example, when $\alpha=3$ it equals $2.618$). Hong also gives some useful estimates on this prefactor. Notice a completely different behaviour of Sugeno integral when compared with Labesgue case. For a literature concerning different inequalities for Sugeno integrals the reader is referred to \cite{Rom13}.

Apart from Sugeno integrals various generalizations of the concept of integration have been proposed, analysed, and applied. Reviewing these would take us too far from the main topic of our short exposition. Some of these generalizations posses their own Bushell-Okrasi\'nski type inequalities. For instance, pseudo-integrals for which, loosely speaking, instead of the field of real numbers one considers a semi-ring defined on a real interval, exhibit a version of (\ref{eqn:BOOrginalStronger}) with redefined multiplication and addition \cite{Dar16}. Notice how different various properties of these integrals might be from the Lebesgue case (like the loss of linearity). But nevertheless, Bushell-Okrasi\'nski inequality (or its variants) remain valid. This strengthens its universal character.

\section{Conclusion}
The Bushell-Okrasi\'nski inequality is a little mathematical gem discovered when studying nonlinear integral equations. The wide array of different possible extensions and generalizations indicate that it is a fundamental relation in mathematical analysis. It waited to be found until almost the end of the twentieth century but now sits comfortably within the collection of its older siblings - Chebyshev, H\"older and Jensen inequalities.

\section*{Acknowledgments}
The author would like to thank Prof. David Edmunds for his invaluable comments and remarks concerning the manuscript. 

\bibliography{biblio.bib}
\bibliographystyle{plain}
	
\end{document}